\documentclass[a4paper]{article}

\usepackage{amsmath,amsthm,amsfonts}
\usepackage[english]{babel}
\usepackage[utf8]{inputenc}
\usepackage{amsmath}
\usepackage{graphicx}
\usepackage{color}

\usepackage{comment}

\newtheorem{theorem}{Theorem}[section]
\newtheorem{definition}[theorem]{Definition}
\newtheorem{proposition}[theorem]{Proposition}
\newtheorem{corollary}[theorem]{Corollary}
\newtheorem{lemma}[theorem]{Lemma}
\newtheorem{problem}[theorem]{Open problem}
\newtheorem*{remark}{Remark}
\newtheorem*{claim}{Claim}

\def\centerby#1#2{%
  \setbox0=\hbox{#1}%
  \hbox to \wd0{\hss#2\hss}%
}

\newcommand{\relpower}[2]{#1^{\circ #2}}
\newcommand{\tuple}[1]{\mathbf{#1}}
\newcommand{\variety}[1]{\mathcal{#1}}

\newcommand{\relstr}[1]{\mathbb{#1}}
\newcommand{\str}[1]{\mathbf{#1}}

\DeclareMathOperator{\Ar}{Ar}

\newcommand{\TC}[1]{\mathcal{#1}}

\begin{document}

\title
      {The weakest nontrivial idempotent equations}

\author{Miroslav Ol\v s\'ak}

\maketitle

\begin{abstract}
An equational condition is a set of equations in an algebraic language, and an algebraic structure satisfies such a condition if it possesses terms that meet the required equations. We find a single nontrivial equational condition which is implied by \emph{any nontrivial idempotent} equational condition.
\end{abstract}

\section{Introduction}

The main result of this paper shows a surprising fact that there is a weakest nontrivial idempotent equational condition for general algebras. In this section we first explain the result in more detail, then discuss the background and motivation, and finally outline the rest of the paper.

\subsection*{Main result}

This paper concerns structures with purely algebraic
signature, so by a \emph{signature} $\Sigma$ we mean a set of operation symbols with associated finite arities; the arity of $f \in \Sigma$ is denoted $\Ar{f}$.
An \emph{algebra} $\str{A}$ of a signature $\Sigma$ consists of a set $A$, called the \emph{universe}, and, for each $f \in \Sigma$, an operation $f^{\str{A}}:A^{\Ar{f}} \to A$ on $A$, called the \emph{basic operation}. 
Each term $t$ in the signature $\Sigma$ over a linearly ordered finite set of variables naturally determines a \emph{term operation} $t^{\str A}$ of $\str A$. 

The next definitions will be illustrated using several terms in a signature that includes a binary operation symbol $\cdot$ and a unary operation symbol $^{-1}$. 
\[
t_1(x,y,z) = (xy)z,\quad
  t_2(x,y,z) = x(yz),\quad
  t_3(x,y,z) = (xy^{-1})z,\quad
  t_4(x,y)   = x.
\]
Note that $t_4$ is also a term in the empty signature and $t_4^{\str A}$ is the binary projection onto the first coordinate. In fact, every term in the empty signature is a variable and the corresponding term operation is a projection.

An \emph{equation} (over a fixed signature) is a pair of terms $s$ and $t$, written $s \approx t$. An algebra $\str{A}$ \emph{satisfies} such an equation if the two terms are evaluated to the same element of $A$ for every evaluation of variables. In other words, $s \approx t$ if  $s^{\str{A}} = t^{\str{A}}$ (for an arbitrary linear ordering of variables). 
For example, every group satisfies the equation $t_1 \approx t_2$.  

An equational condition, informally, stipulates the existence of terms satisfying specified system of equations. More formally,
an \emph{equational condition} $\TC{S}$ is a system of equations in some signature, say $\Delta$. An algebra $\str{A}$ of an arbitrary signature $\Sigma$ is said to \emph{satisfy} $\TC{S}$ if all the equations involved in $\TC{S}$ are satisfied in $\str{A}$ after replacing each operation symbol in $\Delta$ by a term in the signature $\Sigma$. For example, the existence of a \emph{Maltsev term}, ie. the equational condition consisting of two equations
\[
m(x,x,y) \approx y \approx m(y,x,x)
\]
is satisfied in every group, because $m=t_3$ satisfies these equations.

An uninteresting equational condition is the existence of an associative binary operation, ie. the equational condition consisting of a single equation
\[
n(n(x,y),z) \approx n(x,n(y,z)),
\]
because it is satisfied in every algebra by putting $n=t_4$. 
In general,
an equational condition is called \emph{trivial} if it is satisfied in every algebra, equivalently, in an algebra in the empty signature with at least two elements. An example of a nontrivial equational condition is the existence of a Maltsev term since neither of the choices $m(x,y,z)=x$, nor $y$, nor $z$, makes both of the equations true. 

Finally, an equational condition is \emph{idempotent} if, for each operation symbol $f$ appearing in the condition, the \emph{idempotency}, ie.\ the equation $f(x,x, \dots, x) \approx x$, is a consequence of the defining equations. For example, the existence of a Maltsev term is an idempotent term condition, while the associativity is not. 

The main result of this paper shows that there is a weakest nontrivial idempotent equational condition. In fact, several such weakest term conditions are given in Theorem~\ref{thm:equiv}, one of which is stated in the following theorem.

\begin{theorem} \label{thm:intro_main}
The following are equivalent for every algebra $\str{A}$.
\begin{enumerate}
\item $\str{A}$ satisfies a nontrivial idempotent equational condition;
\item $\str{A}$ has a $6$-ary idempotent term $t$ satisfying the  equations
\[
t(xyy,yxx) \approx
 t(yxy,xyx) \approx
 t(yyx,xxy)\,.
 \]
\end{enumerate}
(The variables are grouped together for better readability.)
\end{theorem}

Note that the displayed equational condition is nontrivial, so only the implication ``1 $\Rightarrow$ 2'' is interesting.

\subsection*{Background}

A central concept in universal algebra is a \emph{variety} -- a class of all algebras of a fixed signature that satisfy a given set of equations.
A variety is said to satisfy an equational condition $\TC{S}$ if all its members do (and then the terms satisfying $\TC{S}$ can be chosen uniformly for all algebras in the variety). We remark that 
 there is no essential difference when ``algebra $\str{A}$'' is replaced by ``variety $\variety{V}$'' in the statement of Theorem~\ref{thm:intro_main} since the two versions of the theorem are equivalent by basic universal algebraic results.

Of particular importance  are equational conditions involving finitely many equations, so called \emph{strong Maltsev conditions}, and their countable disjunctions, so called \emph{Maltsev conditions}, since they often characterize structural properties of varieties. The terminology comes from the first characterization of this sort due to A. I. Maltsev (see~\cite{Bergman}) who proved that a variety is congruence permutable (ie., any two congruences in any algebra from $\variety{V}$ permute) if and only if $\variety{V}$ has a Maltsev term.
Another classic Maltsev conditions are those for congruence distributivity due to B. J\'onsson and congruence modularity due to A. Day (see~\cite{Bergman}).

Equational conditions can be preordered by their strength: $\TC{S}$ is weaker than $\TC{T}$, set
$\TC{S} \leq \TC{T}$, if every algebra that satisfies $\TC{T}$ also satisfies $\TC{S}$. By identifying conditions of equal strength, we get a lattice isomorphic to the \emph{lattice of interpretability types of varieties}~\cite{GarciaTaylor}. The main result can be interpreted in its sublattice formed by the idempotent equational conditions: the bottom element (corresponding to the trivial conditions) has a unique upper cover. This is in contrast to the situation in the whole interpretability lattice. W. Taylor~\cite{Taylor88} proved that the bottom element has no cover at all and a general non-covering result was given by R. McKenzie and S. Swierczkowski~\cite{Noncovering}. The first example of a covering in the lattice is due to R. McKenzie~\cite{Covering} who proved that the equations defining Boolean algebras determine a equational condition with a unique upper cover. 

The restriction to idempotent conditions in Theorem~\ref{thm:intro_main}, which is necessary by the mentioned result of W. Taylor, is also quite natural. One reason is that most of the useful Maltsev conditions  are idempotent, including the conditions for congruence permutability, distributivity, and modularity.
Although our results give nontrivial information for some of these conditions, there are several motivations to investigate the algebras  satisfying some nontrivial idempotent equational condition in general.

One of the early appearances of such algebras is in the work of 
W. Taylor~\cite{Taylor77} who studied how equations satisfied by a topological algebra influence group equations obeyed by its homotopy group. One of his results is, roughly, that an equational condition implies some nontrivial group equation if and only if it implies the commutativity of the homotopy groups, and this happens if and only if the equational condition implies a nontrivial idempotent one. A characterization of algebras satisfying a nontrivial idempotent equational condition (see Section~\ref{sect:taylor}), which Taylor gave as a corollary of his results, was later used frequently and is used in this paper as well. 
This motivates the following definition.

\begin{definition}
An algebra is called \emph{Taylor} if it satisfies a nontrivial idempotent equational condition.
\end{definition}

Another significant appearance of Taylor algebras is in the Tame Congruence Theory (TCT) of D. Hobby and R. McKenzie~\cite{HobbyMcKenzie}. The TCT is a structure theory of finite algebras that recognizes 5 types of local behaviors in an algebra and gives ways to deduce global properties from the local ones. The worst, least structured type of behavior is the ``unary type'' and there is a strong correlation of omitting this type and idempotent equations: a finite algebra $\str A$ is Taylor if and only if all finite algebras in the variety generated by $\str A$ omit the unary type. 

A more recent strong motivation to study Taylor algebras in general comes from the fixed--template  Constraint Satisfaction Problem (CSP). The CSP over a  relational structure $\relstr{A}$ (called the template) is a computational problem that asks whether an input primitive positive sentence in the language of $\relstr{A}$ is true in $\relstr{A}$. A lot of recent attention is devoted to understanding how the computational or descriptive complexity of the CSP depends on the relational structure, see~\cite{CSPSurvey} for a recent survey.

For relational structures with finite universes there is a tight connection between the  complexity of the associated CSP and equational conditions for algebras. Namely, the complexity of the CSP over a finite $\relstr{A}$ is fully determined by the equational conditions satisfied by the so called \emph{algebra of polymorphism}, whose basic operations are all the homomorphisms from cartesian powers of $\relstr{A}$ to $\relstr{A}$. Moreover, without loss of generality, it is possible to consider only those structures whose associated algebra $\str A$ is idempotent. Under this assumption, it is known that the CSP is NP--complete whenever $\str A$ is not Taylor and the \emph{algebraic dichotomy conjecture}~\cite{BJK}, confirmed in many special cases, states that the CSP is otherwise solvable in polynomial time. An intensive research motivated by this conjecture has brought a number of strong characterizations of finite Taylor algebras, including the equational condition given by M. Siggers~\cite{Siggers}. refined by K. Kearnes, P. Markovi\'c, and R. McKenzie~\cite{OptimalStrong}: Finite $\str A$ is Taylor if and only if $\str A$ has a $4$-ary idempotent term $s$ satisfying the equation
\[
s(r,a,r,e) \approx s(a,r,e,a)\,.
\]
The fact that there is a weakest idempotent equational condition for finite algebras was  unexpected and possible extension to infinite algebras was not considered until much later, in the context of infinite domain CSPs. 

The CSP over infinite relational structures is also an active research area, see \cite{BodirskySurvey,Bodirsky-HDR-v5} for a survey. In particular,  M. Bodirsky and M. Pinsker (see~\cite{BPP-projective-homomorphisms}) extended the dichotomy conjecture to  a certain class of infinite structures. However, their dividing line involves both equational and topological properties of the polymorphism algebra, which brought the question whether the topological structure is essential in their criterion. 
During the Banff workshop ``Algebraic and Model Theoretical Methods in Constraint Satisfaction'', November 2014, various versions of this problem were discussed and a ``solution'' to one of them emerged from the discussions depending on the ``obvious fact'' that there is no weakest nontrivial equational condition for idempotent algebras. Filling in this gap turned out to be more complex than expected, however, some partial results were obtained, for instance, A. Kazda observed that the rare--area term is not the weakest one in general (see Theorem~\ref{thm:kazda}).  We remark that the original problem, Question 1.3. in~\cite{BPP-projective-homomorphisms}, remains open.
On the other hand, it was proved by L. Barto and M. Pinsker~\cite{BartoPinsker} that the topological structure is indeed irrelevant in the Bodirsky--Pinsker dichotomy conjecture.
An intermediate problem, a ``loop lemma for near unanimity'', which they considered while working on the result, turned out to have positive answer that requires no additional algebraic or topological assumptions. This fact evolved into the main result of the present paper and actually forms a significant part of the proof. 

There does not seem to be any immediate application of the results of this paper to the CSP. However, we believe that the ideas will be useful to address some of these problems, such as those in~\cite{BartoPinsker}.


\subsection*{Outline}

After the preliminaries in Section~\ref{sect:prelim}, we state in Section~\ref{sect:taylor} the mentioned Taylor's characterization of Taylor algebras, discuss further characterizations known in the finite word, and show the difficulties when going infinite. Some of the characterizations of finite Taylor algebras are based on ``loop lemmata'', certain results of combined graph theoretic and algebraic flavor.  An infinite loop lemma is given in Section~\ref{sect:loop}. This loop lemma is then used to prove a ``double loop lemma'' in Section~\ref{sect:double_loop} and a weakest idempotent equational condition is derived as a consequence. Next, in Section~\ref{sect:equiv}, we give a number of equivalent conditions, including the one stated in Theorem~\ref{thm:intro_main}. We finish by discussing open problems in Section~\ref{sect:open}.

\section{Preliminaries} \label{sect:prelim}

In this section, we fix some notation and terminology, and recall some basic  facts. Standard references for universal algebra are \cite{BS,MMT} and a more recent~\cite{Bergman}.

An operation $f$ on a set $A$ is \emph{idempotent} if $f(a,a,\ldots,a)=a$ for any $a \in A$, and an algebra is \emph{idempotent} if all of its basic operations (equivalently, term operations) are idempotent. For convenience we will often formulate definitions and results only for idempotent algebras. For instance, in Theorem~\ref{thm:intro_main} we would assume that $\str A$ is idempotent and omit the other two occurrences of idempotency. The difference is only cosmetic.

An $n$-ary operation $f$ on a set $A$ is \emph{compatible} with an $m$-ary relation $R \subseteq A^m$, or $R$ is \emph{compatible} with $f$, if $f(\tuple{r}_1, \ldots, \tuple{r}_n) \in R$ for any $\tuple{r}_1, \ldots, \tuple{r}_n \in R$. Here (and later as well) we abuse the notation and use $f$ also for the $n$-ary operation on $A^m$ defined from $f$ coordinate-wise.
A subset $B \subseteq A$ is a \emph{subuniverse} of an algebra $\str A$, written $B \leq \str A$, if it is the universe of a subalgebra of $\str A$; in other words, it is compatible (as a unary relation) with every basic operation (equivalently, term operation) of $\str A$. The smallest subuniverse of $\str A$ containing a set $B$ is called the subuniverse  \emph{generated} by $B$. It is equal to 
\[
\{ t(b_1, \dots, b_n): n \in \mathbb{N}, \, b_i \in B, \, t \mbox{ an $n$-ary term operation } of \str A\}.
\]

A stronger compatibility notion, absorption, turned out to be fruitful for finite algebras and finite domain CSPs~\cite{AbsorptionSurvey} and it will be useful in this paper as well.
\begin{definition}
Let $A$ be a set, $X,Y$ subsets of $A$, and  $f$ an $n$-ary operation on $A$.
We say that $X$ \emph{absorbs} $Y$ with respect to $f$ if for
any coordinate  $i=1,\ldots,n$ and any elements
$x_1,x_2,\ldots,x_{i-1},y,x_{i+1},\ldots,x_n\in A$
such that $y\in Y$ and each $x_j\in X$, we have
$$
t(x_1,x_2,\ldots,x_{i-1},y,x_{i+1},\ldots,x_n)\in X.
$$
\end{definition}
\noindent
This concept can be regarded as a generalization of near unanimity operations.
\begin{definition}
An operation $f$ on a set $A$ of arity $n > 2$ is called a \emph{near unanimity} operation, or \emph{NU}, for short, if $f(x,\ldots,x,y,x\ldots,x) = x$ for any $x,y \in A$ and any position of $y$ in the $n$-tuple of arguments.
\end{definition}

We will often work with binary relations (usually symmetric) on a set $A$. We will look at them as graphs and use a graph theoretic terminology. A tuple $(a_1, \dots, a_l) \in A^l$ is an \emph{$R$-walk} of length $l-1$ from $a_1$ to $a_l$  if $(a_i,a_{i+1})\in R$ for all
$i=1,\ldots,l-1$. If, moreover, $a_1=a_l$, we call the $R$-walk an
\emph{$R$-cycle}. A \emph{loop} in $R$ is a pair $(a,a) \in R$. 
The $k$-fold composition of $R$ with itself is denoted by $\relpower{R}{k}$, ie.\  $(a,b)\in \relpower{R}{k}$ if there is an $R$-walk of length $k$ from $a$ to $b$.
The set of out-neighbors of an element $a \in A$ or a set $B \subseteq A$ is denoted by $a^{+R}$ and $B^{+R}$, that is,
\[
a^{+R} = \{b: (a,b) \in R\}, \quad B^{+R} = \bigcup_{b \in B} b^{+R}
\]

Fix a signature $\Sigma$. An equation $s \approx t$ is a \emph{consequence} of a system of equations $\TC{S}$, or $\TC{S}$ \emph{implies} $s \approx t$, if an algebra satisfies $s \approx t$ whenever it satisfies each equation in $\TC{S}$. The consequence relation between equational conditions is defined similarly (see the introduction). We remark that both consequence relations can be equivalently defined in a purely syntactic way. 

The \emph{absolutely free algebra} (in the signature $\Sigma$) over a set of generators $X$ has as its universe the set of all terms over $X$ and basic operations act in the natural way. The \emph{free algebra} over $X$ modulo a set of equations $\TC{S}$ is a quotient of the absolutely free algebra over $X$, where $s$ and $t$ are identified if and only if $s \approx t$ is a consequence of $\TC{S}$. 
Note that an equational condition $\TC{S}$ implies an equational condition $\TC{T}$ if and only if the free algebra over $X$ (with $|X|$ at least the number of variables occurring in $\TC{T}$) modulo $\TC{S}$ satisfies $\TC{T}$.

An equation is \emph{linear} if it involves only terms of height at most one, ie.\ it is of the form 
\[
t(\mbox{variables}) \approx s(\mbox{variables}), \quad \mbox{ or } \quad
t(\mbox{variables}) \approx \mbox{variable},
\]
where $s,t$ are operation symbols. Similarly, a system of equation is \emph{linear} if all of its members are. We will be mostly dealing with linear equations and their systems. The following composition of terms, the \emph{star composition}, is often used to produce linear equational conditions.
\begin{definition} \label{def:star}
Let $f,g$ be terms of arity $n$, $m$, respectively. Then $f*g$
denotes the $(n\times m)$-ary term
\begin{align*}
&(f*g)(x_{1,1},\ldots,x_{1,m},x_{2,1}\ldots x_{n,m})  \\
&\approx f(g(x_{1,1},\ldots x_{1,m}),g(x_{2,1},\ldots x_{2,m}),\ldots
  g(x_{n,1},\ldots x_{n,m}))
\end{align*}
\end{definition}
\noindent
Note that both $f$ and $g$ can be recovered from $f*g$ if they are idempotent.

\section{Taylor algebras} \label{sect:taylor}

The basic tool for us will be a characterization of Taylor algebras by means of Taylor terms.

\begin{definition}
An $n$-ary term $t$ is a \emph{Taylor term} of an idempotent algebra $\str{A}$ if $\str{A}$ satisfies a system of equations in two variables $x,y$ of the form
\begin{align*}
t(x, ?, ?, \ldots, ?) &\approx t(y, ?, ?,\ldots, ?),\\
t(?, x, ?, \ldots, ?) &\approx t(?, y, ?, \ldots, ?),\\
&\centerby{${}\approx{}$}{$\vdots$} \\
t(?, ?, \ldots, ?, x) &\approx t(?, ?, \ldots, ?, y),
\end{align*}
where each question mark stands for either $x$ or $y$. 

Such a system of equations is called a \emph{Taylor system of equations}.
An operation $f$ on a set $A$ is called a \emph{Taylor operation} if it satisfies some system of Taylor equations.
\end{definition}

An example of a Taylor term is a Maltsev term $m$ from the introduction. Indeed, the defining equations $m(x,x,y) \approx y \approx m(y,x,x)$ imply
\begin{align*}
m(x,x,x) &\approx m(y,y,x) \\
m(x,x,x) &\approx m(y,y,x) \\
m(x,x,x) &\approx m(x,y,y)
\end{align*}

No Taylor system of equations is satisfiable by projections since the $i$-th equation prevents $t$ from being a projection to the $i$-th coordinate. Any idempotent algebra with a Taylor term is thus a Taylor algebra. Taylor proved that the converse implication also holds. 

\begin{theorem}[Corollary 5.3 in \cite{Taylor77}]
The following are equivalent for every idempotent algebra $\str A$.
\begin{itemize}
\item $\str{A}$ is a Taylor algebra;
\item $\str{A}$ has a Taylor term.
\end{itemize}
\end{theorem}

Several strengthenings of this theorem for 
\emph{finite} algebras are formulated in the following theorem.

\begin{theorem} \label{thm:finite_taylor}
The following are equivalent for each finite idempotent algebra $\str A$.
\begin{itemize}
\item $\str A$ is a Taylor algebra;

\item \cite{WNU} For some $n \geq 2$, $\str A$ has a term $t$ of arity $n$ that satisfies
\[
t(x,x,\ldots,x,y) \approx t(x,\ldots,x,y,x) \approx \cdots \approx t(x,y,x\ldots,x) \approx t(y,x,\ldots,x,x);
\]
\textup(\emph{weak near unanimity term} of arity $n$, or $n$--WNU for short\textup)

\item \cite{cyclic} For each prime $n > |A|$, $\str A$ has a term $t$ of arity $n$ that satisfies
\[
t(x_1, x_2, \ldots, x_n) \approx t(x_2, \dots, x_n, x_1);
\]
\textup(\emph{cyclic term}\textup)

\item \cite{Siggers}
$\str A$ has a $6$-ary term $t$ that satisfies
$s(x,y,x,z,y,z) \approx s(y,x,z,x,z,y)$;
\textup(\emph{6-ary Siggers term}\textup)

\item \cite{OptimalStrong}
$\str A$ has a $4$-ary term $t$ that satisfies
$s(r,a,r,e) \approx s(a,r,e,a)$.
\textup(\emph{4-ary Siggers term}\textup)
\end{itemize}
\end{theorem}

Note that all the terms that appear in Theorem~\ref{thm:finite_taylor} are Taylor terms although the defining equations of cyclic and Siggers terms involve more than two variables. Two variable equations can be simply obtained by suitable substitution of variables, eg.\ the 4-ary Siggers term implies
\begin{align*}
s(x,y,x,x) &\approx s(y,x,x,y) \\
s(y,x,y,x) &\approx s(x,y,x,x) \\
s(x,y,x,y) &\approx s(y,x,y,y) \\
s(y,y,y,x) &\approx s(y,y,x,y)
\end{align*}

None of the strengthenings of Taylor terms in Theorem~\ref{thm:finite_taylor} work for infinite algebras.
The following algebra can serve as a counterexample for WNUs (or cyclic terms): The universe is the set of all integers and basic operations are all the operations of the form $f(x_1, \dots, x_n) = \sum_{i=1}^n a_ix_i$, where $a_i$'s are integers with $\sum_{i=1}^n a_i = 1$. This algebra is idempotent and has a Maltsev term $m(x,y,z) = x - y + z$. On the other hand, it has no weak near unanimity term since each term operation is a basic operation and no basic operation is a WNU (the WNU equations force $a_1 = a_2 = \cdots = a_n$ but then $\sum_{i=1}^n a_i \neq 1$).

As for Siggers terms, Alexandr Kazda proved that Taylor terms, or even WNU terms, do not imply any nontrivial  strong Maltsev condition involving a single linear equation. We include a sketch of his argument.

\begin{theorem} \label{thm:kazda}
There is an idempotent algebra which has a 3--WNU term but does not satisfy any nontrivial strong Maltsev condition consisting of a single linear equation. 
\end{theorem}

\begin{proof}[Sketch of proof]
Consider the signature consisting of a single ternary symbol $t$ and take the free algebra $\str F$ over countably many generators modulo $\{t(x,x,x) \approx x, t(x,x,y) \approx t(x,y,x) \approx t(y,x,x)\}$. By definition, $\str F$ is idempotent and $t$ is a 3--WNU of $\str F$. 
We define a binary relation $R \subset F^2$ so that $(s,t) \in R$ if no representative of $s$ is a subterm of a representative of $t$ and, conversely, no representative of $t$ is a subterm of a representative of $s$. The relation $R$ is compatible with $t$ but it is not compatible with any operation satisfying a nontrivial linear equation.
\end{proof}

We finish this section with two remarks which say that Theorem~\ref{thm:kazda} is in a sense optimal. The first observation is that any idempotent Taylor algebra satisfies a nontrivial system of two linear equations in a single operation symbol. Indeed, if $t$ is a Taylor term, then
\[
  \vbox{\def\c#1{\hfil#1\hfil}
  \halign{$#{}$&$t(t(\c{#},{}$&$\c{#},\ldots,{}$&$\c{#}),{}$&%
                    $t(\c{#},{}$&$\c{#},{}$&$#$&%
                    $),\ldots,{}t(\c{#},{}$&$#$&$,\c{#}))$&${}#$\cr
 &x_1&x_2&x_n&x_1&x_2&  \ldots,\c{x_n}  &x_1&\c{x_2},\ldots&x_n&\cr
\approx&x_1&x_1&x_1&x_2&x_2&  \ldots,\c{x_2}  &x_n&\c{x_n},\ldots&x_n& \omit,\hss\cr
\noalign{\smallskip}
 & x & ? & ? & ? & x &\c{?},\ldots,\c{?}& ? & \ldots,\c{?} & x &\cr
\approx& y & ? & ? & ? & y &\c{?},\ldots,\c{?}& ? & \ldots,\c{?} & y & \omit,\hss\cr
  }}
\]
where the question marks are chosen in accordance with the Taylor equations. 
These two equations trivially imply two linear equations for $s = t*t$ which form a nontrivial strong Maltsev condition. Note that the first equation follows solely from the idempotency while the second from the Taylor equations. 
This will be a feature of the first weakest nontrivial system of two equations from Section~\ref{sect:double_loop}. 

The second observation is that any idempotent Taylor algebra satisfies a nontrivial nonlinear equation. Indeed, the second equation from those above is nontrivial when considered in the signature $\{t\}$. In particular, our weakest nontrivial conditions can be rewritten into a single nontrivial equation.

\section{A loop lemma} \label{sect:loop}

By a \emph{loop lemma} we mean a statement of the form: If a binary relation satisfies some structural assumption and is compatible with some ``nice'' operations, then it contains a loop (ie., a pair $(a,a) \in R$). An example of a loop lemma is the following theorem. It can be deduced from~\cite{HellNesetril} and in this form it was proved in~\cite{BulatovLoop}.

\begin{theorem} [\cite{HellNesetril,BulatovLoop}] \label{thm:loop_HN}
If $R$ is a symmetric relation on a finite set $A$, $R$ contains an odd cycle, and $R$ is compatible with an idempotent Taylor operation on $A$, then $R$ contains a loop.
\end{theorem}

A generalization of Theorem~\ref{thm:loop_HN}~\cite{TheLoopLemma} (see also~\cite{cyclic}), sometimes referred to as ``the Loop Lemma'', weakens the assumption on $R$: $R$ is \emph{smooth} (ie. a vertex has an incoming edge if and only if it has an outgoing edge) and $R$ has \emph{algebraic length one} (ie. there is a closed walk with one more forward edges than backward edges). 

Both versions were originally used to prove NP-completeness of some CSPs. Later, it was observed that one can apply these results to obtain strong Maltsev conditions for finite Taylor algebras; the 6-ary Siggers term~\cite{Siggers} from Theorem~\ref{thm:loop_HN} and the 4-ary version ~\cite{OptimalStrong} from the mentioned generalization (the terms are obtained in the same way as in Corollary~\ref{cor:NUSiggers}). 

The finiteness assumption in Theorem~\ref{thm:loop_HN} is essential as witnessed by the binary relation in the proof of Theorem~\ref{thm:kazda}. However, an infinite analogue of Theorem~\ref{thm:loop_HN} becomes true when the algebraic assumption is strengthened to ``$R$ is compatible with a near unanimity operation on $A$'', see Corollary~\ref{cor:NUloop}.
Such a loop lemma would be sufficient for our purposes. Nevertheless, in order to isolate the crucial property and for possible future reference, we prove a slightly stronger version which uses the following concept.

\begin{definition}
Let $A$ be a set, $f$ an operation on $A$ and $R\subset A^2$  a symmetric relation. We say that $R$ \emph{produces enough absorption} with respect to $f$ if
 for every element $x\in A^{+R}$ (a non-isolated element), the set $x^{+R}$ of neighbors of $x$ 
 absorbs $\{x\}\cup x^{+R}$ with respect to $f$.
\end{definition}

We are ready to state and prove the promised loop lemma.

\begin{theorem} \label{thm:loop}
Let $A$ be a set,
$R\subset A^2$ a symmetric binary relation containing an odd cycle, and $f$ an operation on $A$ compatible with $R$ such that $R$ produces enough absorption wrt.\  $f$. 
Then $R$ contains a loop. 
\end{theorem}

The theorem immediately follows from the following technical result by putting $g=f$.

\begin{lemma} \label{lem:real_loop}
Let $A$ be a set, $R\subset A^2$ a symmetric binary relation,
$f,g$ operations on $A$, and  $l$ a positive odd integer. Moreover, assume that
\begin{enumerate}
\item[(1)] $R$ contains a cycle of length $l$.
\item[(2)] $R$ is compatible with $f$,
\item[(3)] $R$ produces enough absorption wrt.\ $f$,
\item[(4)] $\Ar{g} \leq \Ar{f}$ and whenever
$(x_1,y_1),\ldots,(x_{\Ar{f}}, y_{\Ar{f}})\in R$,
  then \\
    $(g(x_1,\ldots,x_{\Ar{g}}),f(x_1,\ldots,x_{\Ar{f}}))\in R$.
\item[(5)] $R$ produces enough absorption wrt.\ $g$,
\end{enumerate}
Then $R$ contains a loop.
\end{lemma}

\begin{proof}
The proof proceeds by induction, primarily on $\Ar{g}$, secondarily on $l$.

We start with the base steps. If $l=1$, then $R$ contains a cycle of length one -- a loop.  If $\Ar{g}=1$, pick a
vertex $x\in A^{+R}$. It is absorbed by $x^{+R}$ wrt.\ $g$, so $g(x)\in x^{+R}$, equivalently $x\in g(x)^{+R}$. Since $g(x)^{+R}$ absorbs itself wrt.\ $g$, it is closed under $g$. Thus $g(x)\in g(x)^{+R}$ and we get the loop $(g(x),g(x))\in R$.

Now suppose $l > 1$, $\Ar{g} > 1$ and use the induction hypothesis
for the same $A$, $f$ and $g$ but with $\relpower R3$ instead of $R$ and $l-2$
instead of $l$. The relation $\relpower R3$ is clearly symmetric, the remaining assumptions are verified as follows. 
\begin{enumerate}
\item[(1)] $\relpower R3$ contains a cycle of length $l-2$: If elements $x_1,x_2,\ldots,x_l$ form an $R$-cycle of length
$l$, then $x_1, x_2,\ldots,x_{l-2}$ form an $\relpower R3$-cycle of length
$l-2$.
\item[(2)] $\relpower R3$ is compatible with $f$: Consider pairs $(x_i,y_i)\in \relpower R3$, where $i=1,\ldots,\Ar{f}$.
Then
there are $u_i,v_i\in A$ such that $(x_i,u_i,v_i,y_i)$ is an $R$-walk of length 3. Since $f$ is compatible with $R$, the tuple
$$
(f(x_1,\ldots,x_{\Ar{f}}),
f(u_1,\ldots,u_{\Ar{f}}),
f(v_1,\ldots,v_{\Ar{f}}),
f(y_1,\ldots,y_{\Ar{f}}))
$$
forms an $R$-walk of length 3 and thus $(f(x_1,\dots,x_{\Ar f}),f(y_1,\ldots,y_{\Ar{f}}))$ is in $\relpower R3$, as required.

\item[(3)] $\relpower R3$ produces enough absorption wrt.\ $f$: Assume $y\in A$ is non-isolated and take $x_1,\ldots,x_{\Ar{f}}$ from $y^{+\relpower R3}$ with one possible exception $x_j=y$. In that case, since
$x_j=y$ is a non-isolated element, we can set $v_j=y$ and pick $u_j$ such that
$(x_j,u_j,v_j)$ forms an $R$-walk. For each $i\neq j$ there is an $R$-walk
$(x_i,u_i,v_i,y)$. Then
$$
(f(x_1,\ldots,x_{\Ar{f}}),
f(u_1,\ldots,u_{\Ar{f}}), f(v_1,\ldots,v_{\Ar{f}}), y)
$$
is an
$R$-walk due to assumptions (2) and (3). Therefore $f(x_1,\ldots,x_{\Ar{f}}) \in y^{\relpower R3}$, as required.
\item[(4)] ``$\relpower R3$ is compatible with $g$-$f$'': Consider $x_i,y_i,u_i,v_i$ as in the proof of item (2). Then
$$(g(x_1,\ldots,x_{\Ar{g}}), f(u_1,\ldots,u_{\Ar{f}}),
f(v_1,\ldots,v_{\Ar{f}}), f(y_1,\ldots,y_{\Ar{f}}))$$
is an
$R$-walk by assumptions (4) and (2).
\item[(5)] $\relpower R3$ produces enough absorption wrt.\ $g$: Let  $y\in A$ be non-isolated and $x_1,\ldots,x_{\Ar{g}}$ in $y^{+R}$ with one possible exception $x_j=y$. In that case, since $x_j=y$ is a
non-isolated element, we can set $v_j=y$ and pick $u_j$ such that that
$(x_j,u_j,v_j)$ forms an $R$-walk. For each $i\neq j, i\leq \Ar{g}$ there is
an $R$-walk $(x_i,u_i,v_i,y)$. Finally, for each $i = \Ar{g}+1, \ldots, \Ar{f}$, we pick
$v_i,u_i,x_i$ such that $(x_i,u_i,v_i,y)$ forms an $R$-walk. Then
the sequence
$$(g(x_1,\ldots,x_{\Ar{g}}),f(u_1,\ldots,u_{\Ar{f}}),
f(v_1,\ldots,v_{\Ar{f}}),y)$$
is an $R$-walk by assumptions (4), (2) and (3), and the claim follows.
\end{enumerate}

\medskip
\begin{figure}[!ht]
  \def\svgwidth{10cm}
  \centering
     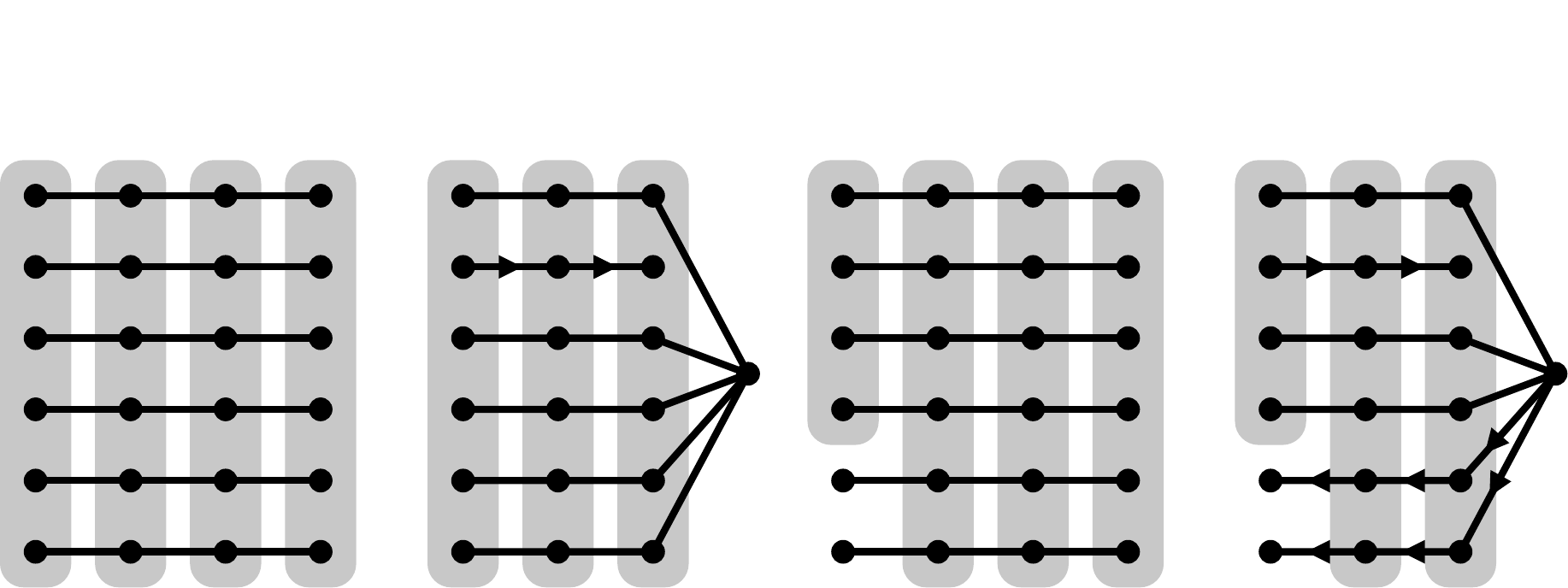
  \caption{The verification of conditions for $\relpower R3$.}
\end{figure}

The induction hypothesis provides a loop in $\relpower R3$, ie.\ a triangle
(cycle of length 3) in $R$. Let us call its vertices $a$, $b$, $c$. We set
$A' = a^{+R}$, so
$b,c\in A'$. Further we put $R' = R\mid_A = R\cap (A')^2, f' = f\mid_{(A')^{\Ar{f}}}$
and define a $(\Ar{g}-1)$-ary operation $g'$ by
$g'(x_1, \ldots x_{\Ar{g}-1}) = g(x_1, \ldots, x_{\Ar{g}-1}, a)$.

A loop will be found within $A'$ using the induction hypothesis for the set $A'$, operations $f', g'$ and the relation $R'$. It remains to verify all the assumptions. The symmetry of $R'$ is again obvious, the rest is seen as follows.
\begin{enumerate}
\item[(0)] $A'$ is closed under the operations $f'$, $g'$: Assume
  $x_1,\ldots,x_{\Ar{f}}\in A'={a}^{+R}$. Since $R$ produces enough
  absorption wrt.\ $f$ and $g$, we have the following.
\begin{align*}
f'(x_1,\ldots,x_{\Ar{f}}) &= f(x_1,\ldots,x_{\Ar{f}})\in A^{+R},\\
g'(x_1,\ldots,x_{\Ar{g}-1})&=g(x_1\ldots,x_{\Ar{g}-1},a)\in A^{+R}.\\
\end{align*}
\item[(1)] $R'$ contains an odd cycle:
The following tuple is an $R$-cycle by the compatibility of $f$ with $R$.
\begin{align*}
(f(b,b,b,\ldots,b),\quad& f(a,c,c,c,\ldots,c),\cr
f(c,b,b,\ldots,b),\quad& f(b,a,c,c,\ldots,c),\cr
f(c,c,b,\ldots,b),\quad& f(b,b,a,c,\ldots,c),\cr
\centerby{$,\quad$}{$\vdots$}& \cr
f(c,c,\ldots,c,b),\quad& f(b,b,b,\ldots,b.a),\cr
f(c,c,\ldots,c,c),\quad& f(b,b,b,\ldots,b.b))
\end{align*}
All the elements of the cycle lie in $A'$ because $A'$ absorbs $A'\cup\{a\}$ wrt.\ $f$.
\item[(2)] $R'$ is compatible with $f'$: Indeed, $f'$ is just a restriction of $f$ compatible with $R$.
\item[(3)] $R'$ produces enough absorption wrt.\ $f'$: Indeed, $f'$ is just a restriction of $f$ and $R$ produces enough absorption wrt.\ $f$.
\item[(4)] ``$R'$ is compatible with $g'$-$f'$'':  Consider pairs $(x_1,y_1),\ldots,(x_{\Ar{f}},y_{\Ar{f}})\in R'$. Then
$$
(x_1,y_1),\ldots,(x_{\Ar{g'}},y_{\Ar{g'}}),
  (a,y_{\Ar{g}}),\ldots,(a,y_{\Ar{f}})\in R,
$$
since $R'\subset R$ and $A' = a^{+R}$. By the original assumption (4),
the element $f(y_1,\ldots,y_{\Ar{f}})$ is an $R$-neighbor of
$g'(x_1,\ldots,x_{\Ar{g'}})=g(x_1,\ldots,x_{\Ar{g'-1}},a)$. Moreover, it is an $R'$-neighbor, since both elements are in $A'$ by (0).
\item[(5)] $R'$ produces enough absorption wrt.\ $g'$:
  Consider an element $y\in A'$ and elements
$x_1, \ldots, x_{\Ar{g}-1}$ such that they are all $R'$-neighbors of $y$ with one possible exception $x_i=y$.
By the original assumption (5) and since $a$ is an $R$-neighbor of $y$, the vertex $z=g(x_1, \ldots, x_{\Ar{g}-1}, a)$ is an $R$-neighbor of $y$. In fact, it is an $R'$-neighbor as $z \in A'$ by (0).
\end{enumerate}
The proof of Lemma~\ref{lem:real_loop} as well as Theorem~\ref{thm:loop} is now concluded.
\end{proof}

\begin{remark}
Ralph McKenzie has found a modification of the proof which does not require the detour through Lemma~\ref{lem:real_loop}. 
He does not keep the original $f$ throughout the proof and instead directly modifies it by plugging $a$ to the last coordinate (in the present proof, this modification is applied to $g$ instead). The new operation is not necessarily compatible with $R$, but it is compatible with $\relpower R3$. This  allows him to produce an arbitrary large clique in $R$, which easily gives the desired loop.
\end{remark}

The following proposition states some sufficient conditions for satisfying the algebraic requirement in Theorem~\ref{thm:loop}. Only the strongest one in item (i) will be used in the next sections.

\begin{proposition} \label{prop:semiabs}
Let $R$ be a symmetric binary relation on a set $A$. Then (i) $\Rightarrow$ (ii) $\Rightarrow$ (iii), and (ii) $\Leftrightarrow$ (ii)'.
\begin{itemize}
\item[(i)] $R$ absorbs $A^2$ wrt.\ an idempotent operation on $A$;
\item[(ii)] $R$ is compatible with a near unanimity operation on $A$;
\item[(ii)'] There exists an operation $f$ compatible with $R$ such that for every $x\in A^{+R}$ the set $x^{+R}$ absorbs $A^{+R}$ wrt.\ $f$.
\item[(iii)] $R$ produces enough absorption wrt.\ a compatible operation $f$ on $A$.
\end{itemize}
\end{proposition}

\begin{proof}
The implication $(ii)'\Rightarrow(iii)$ follows from the definitions. We will prove $(i)\Rightarrow(ii)'$,
$(ii)\Rightarrow(ii)'$ and $(ii)'\Rightarrow(ii)$. 

$(i)\Rightarrow(ii)'$. Let $f$ be an $n$-ary idempotent operation on $A$ such that
$R$ absorbs $A^2$ wrt.\ $f$ (recall we abuse the notation and write $f$ also for the corresponding operation on $A^2$). If $a_1, \ldots, a_n \in x^{+R}$ with a possible exception of $a_i \in A$, then $(x,a_j) \in R$ for every $j$ with a possible exception of $j = i$. But then $x = f(x,\ldots,x)$ is $R$-related to $f(a_1,\ldots,a_n)$ as $R$ absorbs $A^2$ and thus $f(a_1, \dots, a_n) \in x^{+R}$.

The proof of $(ii)\Rightarrow(ii)'$ is similar. 
If $a_1, \ldots, a_n \in x^{+R}$ with a possible exception $a_i \in A^{+R}$, then $f(a_1, \ldots, a_n)$ is $R$-related to $x=f(x,\ldots, x,b,x,\ldots,x)$, where $b$ is a neighbor of $a_i$.

$(ii)'\Rightarrow(ii)$. Let $f$ be as in item $(ii)'$ and let $n$ denote its arity. We may assume
$n\geq3$, otherwise we add redundant arguments to $f$. We modify $f$ in the simplest way to obtain an NU operation: define  $u$ by
$$
u(x,y,y,\ldots,y) = u(y,x,y,y,\ldots,y) = \cdots =
u(y,y,\ldots,y,x) = y,
$$
$$
u(x_1,\ldots,x_n) = f(x_1,\ldots,x_n) \quad\text{in all the remaining cases}
$$
It is straightforward to verify that $u$ is compatible with $f$.
\end{proof}

An immediate consequence of  Proposition~\ref{prop:semiabs} and Theorem~\ref{thm:loop} is a loop lemma for NU.

\begin{corollary} (Loop lemma for NU)  \label{cor:NUloop}
If $R$ is a symmetric relation on a set $A$, $R$ contains an odd cycle, and $R$ is compatible with a near unanimity operation  on $A$, then $R$ contains a loop.
\end{corollary}

The proof of the final corollary in this section shows how equational conditions are derived from loop lemmata. 

\begin{corollary} \label{cor:NUSiggers}
Every algebra with a near unanimity term has a 6-ary Siggers term. 
\end{corollary}
\begin{proof}
Let $\str F$ be the free algebra over $\{x,y,z\}$ modulo the
NU equations. Let $R$ be the subalgebra of $\str F^2$ generated by the pairs \[(x,y),(y,x),(x,z),(z,x),(y,z),(z,y).\] 
Since the generators form a symmetric graph with an odd cycle, $R$ is symmetric and contains an odd cycle. By definition, $R$ is compatible with an NU operation. Therefore, by Corollary~\ref{cor:NUloop}, $R$ contains a loop $(a,a)$. This loop can be obtained from the generators by a term operation $t^{\str F^2}$, that is,
\[
t^{\str F^2}\left( (x,y), (y,x), (x,z), (z,x), (y,z), (z,y)\right) = (a,a), 
\]
and thus
$t^{\str F}(x,y,x,z,y,z) = t^{\str F}(y,x,z,x,z,y)$. By the definition of free algebras, this means that $t(x,y,x,z,y,z) \approx t(y,x,z,x,z,y)$ in $\str F$.
We have proved that the free algebra on three generators modulo the NU equations has a 6-ary Siggers term and the claim follows.
\end{proof}

\section{Double loop lemma and double loop terms} \label{sect:double_loop}

Armed by Theorem~\ref{thm:loop}, we are ready to prove that a Taylor term implies a specific $12$-ary Taylor term introduced in the next definition.

\begin{definition} \label{def:double_loop}
An $12$-ary term $d$ is a \emph{double loop term} of an idempotent algebra $\str A$ if $\str A$ satisfies the equations
\[
d(xx,xxxx,yyyy,yy) \approx
d(xx,yyyy,xxxx,yy)
\]
\[
d(xy,xxyy,xxyy,xy) \approx
d(yx,xyxy,xyxy,yx)
\]
\end{definition}

The double loop equations can be obtained as follows. Consider a $4 \times 12$ matrix whose columns are all the four-tuples $(a_1,a_2,b_1,b_2) \in \{x,y\}^4$ with $a_1 \neq a_2$ or $b_1 \neq b_2$, and let $\tuple{r}_1,\tuple{r}_2,\tuple{r}_3,\tuple{r}_4$ denote its rows. The double loop equations are then $d(\tuple{r}_1) \approx d(\tuple{r}_2)$ and $d(\tuple{r}_3) \approx d(\tuple{r}_4)$. If the columns are organized lexicographically with $x < y$, we get the equations in Definition~\ref{def:double_loop}. 

Observe that a double loop term is a Taylor term, because the four columns $(a_1,a_2,b_1,b_2)$ with $a_1=a_2$ and $b_1=b_2$ are missing. Conversely, any nontrivial system of two linear equations in one operation symbol and two variables $x,y$ comes from a $4 \times n$ that omit these four columns. Note that each such a system implies a double loop term. Indeed, if some columns are repeated, we can identify variables and get a term whose matrix has  non-repeating columns. Then a double loop term is obtained by introducing dummy variables and reordering the arguments if necessary. 
In this sense, the double loop system of equations is the weakest Taylor system of two equations. 

A double loop term will be derived from a Taylor term using a double loop lemma (Theorem~\ref{thm:double_loop} below), in a similar way in which Siggers term was derived from the NU loop lemma in Corollary \ref{cor:NUSiggers}. In fact, the first equation will be a consequence of idempotence alone, while the second equation will use only the Taylor equations without the idempotency equation.

\begin{theorem} (Double loop lemma) \label{thm:double_loop}
Let  $\str A = (A; t^\str A)$ and $\str B = (B; t^\str B)$ be algebras in the signature consisting of a single $n$-ary operation symbol $t$.
Assume that $\str A$ is generated by $\{x^\str A,y^\str A\}$, $t^\str A$ is idempotent, $\str B$ is generated
by $\{x^\str B,y^\str B\}$ and $t^\str B$ is a Taylor operation. 
Let $Q$ be the subuniverse of $\str A^2\times \str B^2$ generated by
all the 12 quadruples $(a_1,a_2,b_1,b_2)$ with
$a_1,a_2\in\{x^\str A,y^\str A\},\; b_1,b_2\in\{x^\str B,y^\str B\}$, and $a_1\neq a_2$ or $b_1\neq b_2$.
Then there is a double loop in $Q$, ie.\ a quadruple
$(a,a,c,c)\in Q$.
\end{theorem}

\begin{proof}
The majority of the proof is devoted to constructing a binary relation $R \leq \str A^2$ and proving the properties (1) through (4) below.  Afterwards, we will finish the proof by applying Theorem~\ref{thm:loop} to $R$.
\begin{enumerate}
\item[(1)]  $R$ is symmetric,
\item[(2)]  $(x^\str A,y^\str A)\in R$,
\item[(3)]  Whenever $(a_1,a_2)\in R$ there exists  $c\in B$ such that $(a_1,a_2,b,b)\in Q$,
\item[(4)]  $R$ absorbs $\str A^2$ wrt. $t^{\str A}$.
\end{enumerate}

We start by recursively constructing a sequence $s_i$, $s'_i$ of elements of $B$. As the first step, let
$$
s_0 = x^\str B,\quad s'_0 = y^\str B.
$$
Let $j=1,\ldots,n$, let $k$ be a non-negative integer and let $e_j$ be the binary term $e_j(x,y)=t(?,\ldots,?,x,?,\ldots,?)$ that appear, say, on the left hand side of the $j$-th Taylor equation for $t^{\str B}$.  We set
$$
s_{kn+j} = e_j^{\str B}(s_{kn+j-1}, s'_{kn+j-1}),\quad 
s'_{kn+j} = e_j^{\str B}(s'_{kn+j-1}, s_{kn+j-1}),
$$
Note that the definition of $s'$ differs from the definition of $s$ just in swapping the roles of $x^{\str B}$ and $y^{\str B}$.

Next, we define binary relations on $A$ by
\begin{align*}
R_i &= \{(a_1,a_2)\in A^2; (a_1,a_2,s_i,s_i)\in Q\text{ and
}(a_1,a_2,s'_i,s'_i)\in Q \}\\
R'_i &= \{(a_1,a_2)\in A^2; (a_1,a_2,s_i,s'_i)\in Q\text{ and
}(a_1,a_2,s'_i,s_i)\in Q \}.
\end{align*}
Finally, we set
$$R' = \bigcup R'_i, \quad R = \bigcup R_i.
$$

\begin{claim}
$R_0\subseteq R_1\subseteq R_2\subseteq \cdots$ and $R'_0\subseteq R'_1\subseteq R'_2\subseteq\cdots$.
\end{claim}
To prove the first part, consider any $(a_1,a_2)\in R_i$ where $i = kn+j-1$.
Then
$$(a_1,a_2,s_i,s_i),(a_1,a_2,s'_i,s'_i)\in Q,$$
so, since $e_j^{\str A}$ is idempotent and $Q$ is a subuniverse of $\str A^2 \times \str B^2$,
$$
(a_1,a_2,s_{i+1},s_{i+1}) = (e_j^{\str A}(a_1,a_1),e_j^{\str A}(a_2,a_2),e_j^{\str B}(s_i,s'_i),e_j^{\str B}(s_i,s'_i))\in Q \mbox{ and}
$$
$$
(a_1,a_2,s'_{i+1},s'_{i+1}) = (e_j^{\str A}(a_1,a_1),e_j^{\str A}(a_2,a_2),e_j^{\str B}(s'_i,s_i),e_j^{\str B}(s'_i,s_i))\in Q.
$$
Therefore $(a_1,a_2)\in R_{i+1}$. The second part is analogous.

\begin{claim} $R$ and $R'$ are subuniverses of $\str A^2$.
\end{claim}
To prove that $R$ is compatible with $t^{\str A}$, let  $(a_1,b_1),\ldots,(a_n,b_n)$ be arbitrary pairs from $R$. By the previous claim, all these pairs belong to $R_k$ for some $k$. We set
$$
a=t^\str A(a_1, \ldots, a_n),\quad b=t^\str A(b_1,\ldots, b_n)
$$
and aim to show that $(a,b) \in R_{k+1}$.
Pick $j$ such that $s_{k+1} = e_j(s_k, s'_k)$ and choose
$c_1,\ldots,c_n\in\{s_k,s'_k\}$ in such a way that $t^\str B(c_1,\ldots,c_n) = e_j^{\str B}(s_k, s'_k) = s_{k+1}$. Denoting $c'_i$ the other element of $\{s_k,s'_k\}$, we also have $t^\str B(c'_1,\ldots,c'_n) = e_j^{\str B}(s'_k, s_k) = s'_{k+1}$. By definition of $R_k$, the subuniverse $Q \leq \str A^2 \times \str B^2$ contains the quadruples $(a_i, b_i, c_i, c_i)$, $(a_i, b_i, c'_i, c'_i)$, therefore it also contains the quadruples $(a,b,s_{k+1},s_{k+1})$, $(a,b,s'_{k+1},s'_{k+1})$ obtained by applying $t^{\str A^2 \times \str B^2}$. Thus $(a,b)\in R$, as claimed. The second part is similar.

\begin{claim} $R' = A^2$
\end{claim}

Consider an arbitrary pair $(a_1,a_2) \in A^2$. Since $\str A$ is generated by $x^{\str A}$ and $y^{\str A}$, there exists a binary term operation $s^{\str A}$ such that $s^{\str A}(x^{\str A},y^{\str A}) = a_1$. Note that $(x^{\str A},x^{\str A})$ and $(y^{\str A},x^{\str A})$ are in $R'_0 \subseteq R'$. As $R'$ is compatible with $s^{\str A}$ and $s^{\str A}$ is idempotent, we get $(a_1,x^{\str A}) = (s^{\str A}(x^{\str A},y^{\str A}),s^{\str A}(x^{\str A},x^{\str A}))\in R'$ and, analogously, $(a_1,y^{\str A}) \in R'$. A similar argument now shows that $(a_1,a_2) \in R$, finishing the proof of the claim.


We are ready to verify the properties (1) through (4) of the relation $R$.
\begin{enumerate}
\item[(1)]  $R$ is symmetric: The mapping $\psi\colon A^2\times B^2\to A^2\times B^2$ swapping
  the first two coordinates of $A$ is an automorphism of $\str
  A^2\times\str B^2$ which preserves the set of generators of $Q$. Therefore, $\psi$ also preserves $Q$. The claim now follows -- witnesses $w, w'\in Q$ for 
  $(a,b)\in R$ are mapped by $\psi$ to witnesses of $(b,a) \in R$.
\item[(2)]  $(x^\str A,y^\str A)\in R$: Indeed, $(x^\str A,y^\str A)\in R_0 \subseteq R$.
\item[(3)]  Whenever $(a_1,a_2)\in R$ there exists $c\in B$ such that $(a_1,a_2,c,c)\in Q$: This follows from the definition of $R$.
\item[(4)]  $R$ is absorbing $A^2$ wrt.\ $t^{\str A}$: Consider $a_1,a_2,\ldots a_n, b_1,b_2,\ldots b_n\in A$ and $j\in\{1,2,\ldots,n\}$ such that for all $j'\neq j$, $R$ contains $(a_j,b_j)$. We claim that $a = t_A(a_1,a_2,\ldots,a_n)$ is $R$-related to $b = t_A(b_1,b_2,\ldots,b_n)$. Pick $i$ of the form $kn+j-1$ and large enough so that for all $j' = 1,2,\ldots,n$,
  $(a_{j'},b_{j'})\in R'_i$ and if $j'\neq j$ also
  $(a_{j'},b_{j'})\in R_i$. We apply $t^{\str A^2 \times \str B^2}$ to an $n$-tuple of quadruples
  in $Q$ of the form
$$
(a_1,b_1,s^?_i,s^?_i), (a_2,b_2,s^?_i,s^?_i), \ldots,
(a_j,b_j,s_i,s'_i), \ldots, (a_n,b_n,s^?_i, s^?_i),
$$
  where each $s^?_i$ is either $s_i$ or $s'_i$. By the $j$-th Taylor equation for $t^{\str B}$, the question marks can be chosen in such a way that both the third and fourth coordinates of the result are equal to $e_j^{\str B}(s_i,s'_i) = s_{i+1}$. Therefore $(a,b,s_{i+1},s_{i+1})\in Q$. Similarly, we get $(a,b,s'_{i+1},s'_{i+1})\in Q$ and thus
 $(a,b)\in R_{i+1}\subset R$.
\end{enumerate}

To finish the proof we want to apply Theorem~\ref{thm:loop} to the relation $R$.
Since $R$ is symmetric and, by (4) and Proposition~\ref{prop:semiabs}, $R$ produces enough absorption
wrt.\ $t^{\str A}$, it remains to verify that $R$ contains an odd cycle.
But this is a simple consequence of (4) -- the following sequence is an $R$-cycle of length $2n-1$:
\begin{align*}
(t^\str A(x^\str A,x^\str A,\ldots,x^\str A),\quad& t^\str A(x^\str A,y^\str A,y^\str A,\ldots,y^\str A),\cr
t^\str A(y^\str A,x^\str A,\ldots,x^\str A),\quad& t^\str A(x^\str A,x^\str A,y^\str A,\ldots,y^\str A),\cr
\centerby{$,\quad$}{$\vdots$}&\cr
t^\str A(y^\str A,\ldots,y^\str A,x^\str A),\quad& t^\str A(x^\str A,x^\str A,x^\str A\ldots,x^\str A))
\end{align*}
Theorem~\ref{thm:loop} produces a loop $(a,a) \in R$ which in turn implies $(a,a,c,c) \in Q$ by property (3).
\end{proof}

\begin{corollary}
An idempotent algebra is Taylor if and only if it has a double loop term.

Moreover, for every Taylor system of equations in an operation symbol $\{t\}$, there is a term $d$ over the signature $\{t\}$ such that the first double loop equation is a consequence of $t(x,x, \dots,x) \approx x$ and the second double loop equation is a consequence of the given Taylor system. 
\end{corollary}

\begin{proof}
As discussed, a double loop term is a Taylor term, so it is enough to verify the second claim. Its proof is similar to Corollary~\ref{cor:NUSiggers}.
For a given system $\TC{S}$ of Taylor equations in the signature $\{t\}$, let $\str A$ be the free algebra $\str A$ over $\{x^\str A,y^\str A\}$ modulo $\{t(x,x, \ldots,x) \approx x\}$ and let $\str B$
be the free algebra over $\{x^\str B,\str y^\str B\}$ modulo $\TC{S}$.
Finally, let $Q$ be the subuniverse of $\str A^2 \times \str B^2$ described in the statement of Theorem~\ref{thm:double_loop}. Then a term $d$ that computes the double loop $(a,a,c,c)$ from the generators is the required term. \end{proof}

\section{Equivalent conditions} \label{sect:equiv}

We have just proved that every Taylor algebra contains a double loop term.  Now we will introduce further nontrivial strong Maltsev conditions implied by (and thus equivalent to) the existence of a double loop term. 

The \emph{strong double loop} equations are similar to the double loop equations but all four expressions are required to be equal, not just equal in pairs, that is,
\begin{align*}
&\mathrel{\hphantom{\approx}} d(xx,xxxx,yyyy,yy) \\
&\approx d(xx,yyyy,xxxx,yy) \\
&\approx d(xy,xxyy,xxyy,xy) \\
&\approx d(yx,xyxy,xyxy,yx).
\end{align*}

These equations can be further strengthened to the
\emph{weak 3--cube} equations in a $6$-ary symbol $t$:
\begin{align*}
 &\mathrel{\hphantom{\approx}} t(xyy,yxx) \\
 &\approx t(yxy,xyx) \\
 &\approx t(yyx,xxy).
\end{align*}

It was known before that each Taylor system of equations imply a nontrivial 
system of linear equations involving ternary symbols. From the double loop equations we obtain \emph{terminator} equations
$$
\vbox{\halign{\hfil$#$&${}#$\hfil\quad&\hfil$#$&${}#$\hfil\cr
c(x,y,x) &\approx c_1(x,x,y), & c(y,x,x) &\approx c_2(x,x,y), \cr
c_i(x,y,x) &\approx c_{i1}(x,x,y), & c_i(y,x,x) &\approx c_{i2}(x,x,y)\hbox{, where }i\in\{1,2\}, \cr
c_{i1}(x,y,x) &\approx c_{i2}(x,y,x), & c_{i1}(y,x,x) &\approx c_{i2}(y,x,x)\hbox{, where }i\in\{1,2\}.\cr
}}
$$
and from the strong double loop equations we moreover get $c_{11}(y,x,x) \approx c_{22}(x,y,x)$, the \emph{strong terminator} terms. A motivation for this condition comes from infinite domain CSP, see the next section. 

\medskip
\begin{figure}[!ht]
  \def\svgwidth{7cm}
  \centering
     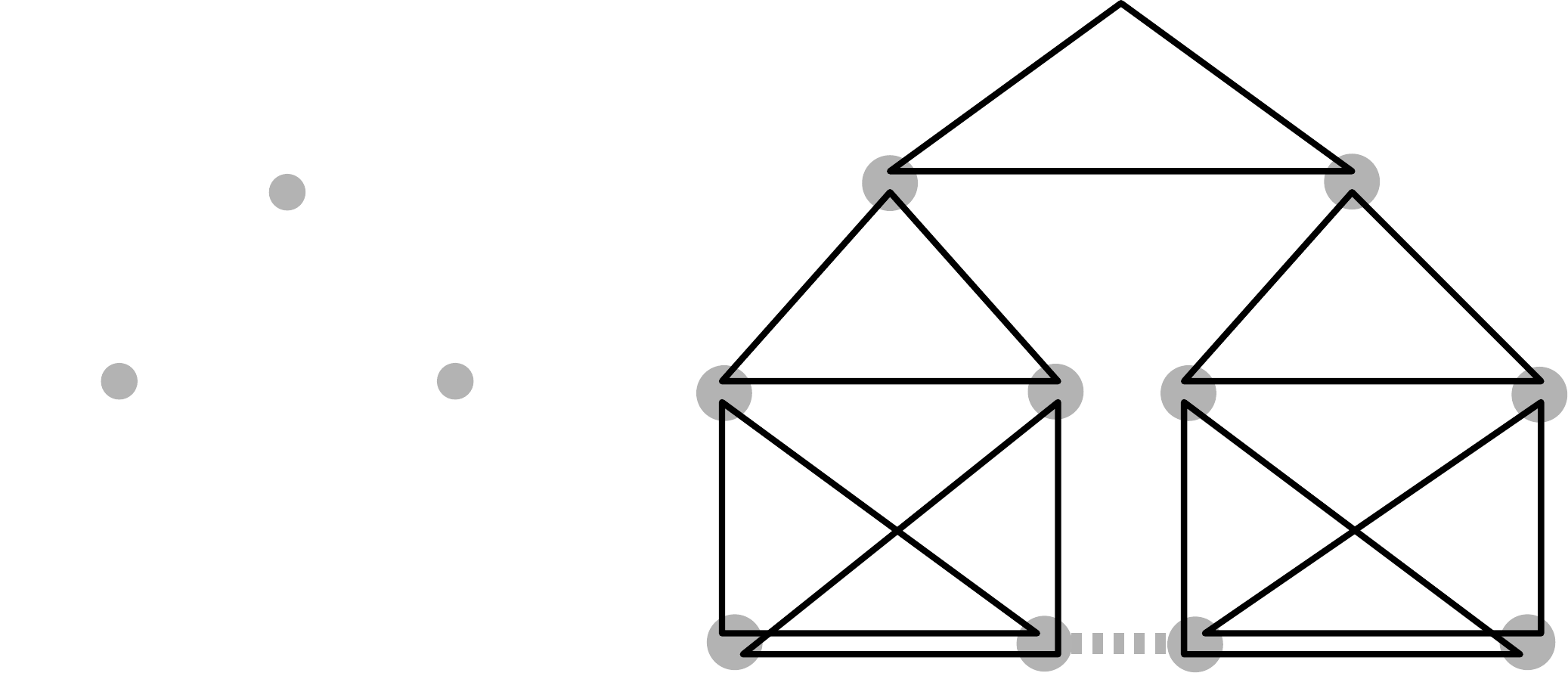
  \caption{Terminator terms; striped connection means the strong variant.}
\end{figure}

\begin{theorem} \label{thm:equiv}
The following are equivalent for every idempotent algebra $\str A$.
\begin{enumerate}
\item[(1)] $\str A$ is a Taylor algebra. 
\item[(2)] $\str A$ has a double loop term.
\item[(3)] $\str A$ has a strong double loop term.
\item[(4)] $\str A$ has a weak 3--cube term.
\item[(5)] There are 4-ary terms $q_1,q_2$ and a ternary term $c$ in $\str A$ satisfying 
$$
\vbox{\halign{\hfil$#$&${}#$\hfil&\hfil${}#$&${}#$\hfil\cr
q_1(x,y,x,y) &\approx q_1(y,x,x,y) &\approx q_2(x,y,x,y) &\approx q_2(y,x,x,y).\cr
q_1(x,x,y,y) &\approx c(x,y,x),    &        q_2(x,x,y,y) &\approx c(y,x,x).\cr
}}
$$
\item[(6)] $\str A$ has terminator terms.
\item[(7)] $\str A$ has strong terminator terms.
\end{enumerate}
\end{theorem}
\begin{proof}
The equivalence $(1)\Leftrightarrow(2)$ is proved in the previous section, trivially 
$(3)\Rightarrow(2)$ and $(7)\Rightarrow(6)$, and (3) or (4) implies (1) since these conditions are nontrivial.

We will prove the implications 
$(2)\Rightarrow(3)$, $(3)\Rightarrow(4)$, $(3)\Rightarrow(5)$, $(5)\Rightarrow(7)$ and that the terminator system is nontrivial, ie.\ $(6)\Rightarrow(1)$.

Let us remark that (3) can be easily deduced from (4) and  that the derivation of strong terminator terms from a strong double loop term, which will be shown, leads as well to terminator terms from a double loop term. However, these implications are not necessary for the proof.
\smallskip

To prove $(2)\Rightarrow(3)$, assume that $\str A$ has a double loop term $d$.
Let us denote by $e(x,y)$ and $f(x,y)$ the terms appearing on (say) the left hand side of the first and second double
loop equations, respectively. Further let $e_1[i]$, where
$i = 1,\ldots,12$, denote the variable at the position $i$ on the left
hand side of the first double loop equation. Similarly, we define $e_2[i]$
for the right hand side and $f_1[i], f_2[i]$ for the second equation.
Finally, we define an operation $\oplus$ on variables $x$, $y$ by
$x\oplus x = y\oplus y = x$ and $x\oplus y = y\oplus x = y$.

Now we describe four ways (a),(b),(c),(d) to substitute the variables of $d*d*d$ by $x$ and $y$ so that the resulting term operation of $\str A$ is equal to $e(f(x,y),f(y,x))^{\str A}$.
Similarly as in Definition~\ref{def:star}, we denote the variables of $d*d*d$ by $x_{i,j,k}$ so that the inner-most $d$'s are applied to $x_{i,j,1}, \ldots, x_{i,j,12}$, etc. The variable $x_{i,j,k}$ is substituted by $x$ or $y$ by the following rules.
$$
(a)\; e_1[j]\oplus f_1[k],\quad
(b)\; e_2[j]\oplus f_1[k],\quad
(c)\; f_1[j]\oplus e_1[i],\quad
(d)\; f_2[j]\oplus e_1[i].
$$
We need to show that in each case, the resulting term evaluates in $\str A$ to $e^{\str A}(f^{\str A}(x,y),f^{\str A}(y,x))$.
In case (a), the inner-most applications of $d^{\str A}$ produces either $f^{\str A}(x,y)$ (if $e_1[j]=x$) or $f^{\str{A}}(y,x)$ (if $e_1[j]=y$). At the middle level, we get $e^{\str A}(f^{\str A}(x,y),f^{\str A}(y,x))$ and the outer-most $d^{\str A}$ does not change the result by idempotency. Case (b) is similar.
In case (c), the inner-most application of $d^{\str A}$ gives $x$ or $y$ by idempotency, the middle level produces $f^{\str A}(x,y)$ or $f^{\str A}(y,x)$ and the outer-most $d^{\str A}$ gives the required result. The last case is, again, analogous.

Observe that for each variable $x_{i,j,k}$, either the substitutions (a) and (b) are different, or (c) and (d). Therefore $t=d*d*d$ satisfies a system of linear equations in two variables of the from $t(\tuple{r}_1) \approx t(\tuple{r}_2) \approx t(\tuple{r}_3) \approx t(\tuple{r}_4)$, where the $\tuple{r}_i$'s are rows of a $4$-row matrix that does not contain the columns $(x,x,x,x)$, $(y,y,y,y)$, $(x,x,y,y)$, $(y,y,x,x)$. Then a strong double loop term can be obtained from $t$ by identification of variables -- see the discussion after Definition~\ref{def:double_loop}.

\smallskip

$(3)\Rightarrow(4)$. Let $\str F$ be the free algebra in the signature $\{d\}$
over $\{x,y\}$ modulo the idempotency and the strong double loop
equations. It suffices to find  a weak 3-cube term in $\str F$ and in order to do that, it is enough to prove that the subuniverse $Q$ of $\str{F}^3$ generated by
$$
\def\vec#1#2#3{\begin{pmatrix} #1\cr #2\cr #3\end{pmatrix}}
\vec xyy
\vec yxy
\vec yyx
\vec yxx
\vec xyx
\vec xxy.
$$
contains a constant triple. (Here it is convenient to write the triples in $Q$ as column vectors.)

Let $\phi$ be the unique automorphism of $\str F$ swapping $x$ and $y$.

\begin{claim}
If $(a,b,c) \in F^3$ is such that $c=\phi(b)$, then $(a,b,c) \in Q$.
\end{claim}

To see this, observe first that $Q$ contains $(a,x,y)$ and $(a,y,x)$
since $(a,x,y)$ can be obtained by applying a binary term operation to
the generators $(x,x,y)$, $(y,x,y)$ and similarly for $(a,y,x)$. Now
any tuple $(a,b,c)$ with $c=\phi(b)$ can be obtained by applying a
term to $(a,x,y)$, $(a,y,x)$.

\begin{claim}
\label{idempotency strength}
Let $\cdot$ be an idempotent term operation of $\str F$. Then there exist $ x_1,  y_1\in F$ such that
\begin{itemize}
\item $y_1= \phi(x_1)$, and
\item the triple $\begin{pmatrix}(y_1x_1)(x_1y_1)\cr x_1\cr x_1\end{pmatrix}$ is in $Q$, where we write $z_1z_2$ instead of $z_1 \cdot z_2$ for brevity.
\end{itemize}
\end{claim}
We set
$$
x_1 = ((xy)x)(y(xy)),\quad
y_1 = ((yx)y)(x(yx))
$$
The first condition is obviously satisfied, the second one is apparent from the following expansion of the triple
$((y_1x_1)(x_1y_1), x_1x_1, x_1)$.
$$
\vbox{\def\c#1{\hfil{#1}\hfil}
 \halign{$[((\c{#}$&$\c{#})$&$\c{#})$&\kern.3em
          $(\c{#}$&$(\c{#}$&$\c{#}$))]&\kern.3em
          $[((\c{#}$&$\c{#})$&$\c{#})$&\kern.3em
          $(\c{#}$&$(\c{#}$&$\c{#}))]$\cr
     (yx) &   y  & (x(yx)) & ((xy)x) & y & (xy) & (xy) & x & (y(xy)) & ((yx)y) &   x  & (yx) \cr
     (xy) & (xy) &    x    &    y    & x &  y   &  x   & y &   x     &    y    & (xy) & (xy) \cr
      x   &   x  &    y    &    x    & x &  x   &  y   & y &   y     &    x    &  y   &   y  \cr
  }}
$$

\begin{remark}
Observe that the claim only requires the idempotency of $\cdot$. It was surprising for us that the simple idempotency equation is actually quite strong. Is there a more conceptual generalization?
\end{remark}

Returning back to the proof of (3) $\Rightarrow$ (4), we apply the claim to the binary operation $xy = d^{\str F}(xx,xxxx,yyyy,yy)$ and obtain $x_1, y_1$ as in the statement.
Let $x_2 = (x_1y_1)(y_1x_1)$, $y_2 = \phi(x_2) = (y_1x_1)(x_1y_1)$.
We claim that the following six triples are in $Q$.

$$
\def\vec#1#2#3{\begin{pmatrix}#1_2\cr#2_1\cr#3_1\end{pmatrix}}
\vec xyy
\vec yxy
\vec yyx
\vec yxx
\vec xyx
\vec xxy.
$$
Indeed, the forth triple is in $Q$ by the claim. The first triple is in $Q$ since it is the $\phi$-image of the forth one, and $Q$ is compatible with $\phi$ (the generators are). For the remaining triples, we can use the first claim.

Finally, let
$z = d^{\str F}(y_2y_2,x_2x_2x_2y_2,
x_2x_2x_2y_2,x_2x_2)$.
The following triples are in $Q$.
$$
  \vbox{\def\c#1{\hfil#1\hfil}
  \halign{$\c{#}={}$&$d(\c{#}$&$\c{#},{}$&$\c{#}$&$\c{#}$&$\c{#}$&$\c{#},{}$&%
          $\c{#}$&$\c{#}$&$\c{#}$&$\c{#},{}$&$\c{#}$&$\c{#})$\cr
         z              & y_2&y_2 & x_2&x_2&x_2&y_2 & x_2&x_2&x_2&y_2 & x_2&x_2 \cr
    (x_1y_1)  & x_1&x_1 & x_1&x_1&x_1&x_1 & y_1&y_1&y_1&y_1 & y_1&y_1 \cr
    (x_1y_1)  & x_1&x_1 & y_1&y_1&y_1&y_1 & x_1&x_1&x_1&x_1 & y_1&y_1 \cr
  }}
$$
$$
  \vbox{\def\c#1{\hfil#1\hfil}
  \halign{$\c{#}={}$&$d(\c{#}$&$\c{#},{}$&$\c{#}$&$\c{#}$&$\c{#}$&$\c{#},{}$&%
          $\c{#}$&$\c{#}$&$\c{#}$&$\c{#},{}$&$\c{#}$&$\c{#})$\cr
         z              & y_2&y_2 & x_2&x_2&x_2&y_2 & x_2&x_2&x_2&y_2 & x_2&x_2 \cr
    (y_1x_1)  & y_1&x_1 & y_1&y_1&x_1&x_1 & y_1&y_1&x_1&x_1 & y_1&x_1 \cr
    (y_1x_1)  & x_1&y_1 & y_1&x_1&y_1&x_1 & y_1&x_1&y_1&x_1 & x_1&y_1 \cr
  }}
$$

Since $z$ is generated by $(x_1y_1)$ and $(y_1x_1)$, then $(z,z,z) \in Q$.

\smallskip

\noindent $(3)\Rightarrow(5)$
Set
\begin{align*}
c(x,y,z)     &= d(y,y,x,z,z,x,x,z,z,x,y,y),\\
q_1(u,v,x,y) &= d(x,x,u,u,u,u,v,v,v,v,y,y),\\
q_2(u,v,x,y) &= d(u,v,x,u,v,y,x,u,v,y,u,v).\\
\end{align*}
The verification of the equations is straightforward.

\smallskip

\noindent $(5)\Rightarrow(7)$
Term $c$ is the same, further put

$$
\vbox{\halign{\hfil$#$&${}#$\hfil\quad&\hfil$#$&${}#$\hfil\cr
   c_1(x,y,z) &= q_1(x,y,z,z), &    c_2(x,y,z) &= q_2(x,y,z,z),\cr
c_{11}(x,y,z) &= q_1(x,z,y,x), & c_{21}(y,x,z) &= q_2(x,z,y,x),\cr
c_{12}(x,y,z) &= q_1(z,x,y,x), & c_{22}(y,x,z) &= q_2(z,x,y,x).\cr
}}
$$
\smallskip

\noindent $(6)\Rightarrow(1)$
Suppose for a contradiction that each term symbol in the terminator system
represents a projection. Let $\pi_1, \pi_2, \pi_3$ denote the
ternary projection to the first, second, third coordinate respectively. Take $i\in\{1,2\}$. 
If $c_i = \pi_1$, we get
$c_{i2} = \pi_3$
by $y \approx c_i(y,x,x) \approx c_{i2}(x,x,y)$. But then $c_{i1}$ can not be equal to $\pi_1,\pi_2$, nor $\pi_3$
because of the equations
\begin{align*}
 c_{i1}(y,x,x) &\approx c_{i2}(y,x,x) = x,\cr
 c_{i1}(x,y,x) &\approx c_{i2}(x,y,x) = x,\cr
x = c_i(x,y,x) &\approx c_{i1}(x,x,y).\cr
\end{align*}
Therefore $c_i\neq \pi_1$.

Analogously, $c_i \neq \pi_2$, therefore $c_1=c_2=\pi_3$. Finally, the equations
$$c(x,y,x) \approx c_1(x,x,y) = y,\quad c(y,x,x) \approx c_2(x,x,y) = y$$
cannot be satisfied by a projection.
\end{proof}

\begin{figure}[!ht]
  \centering
   \includegraphics[width=8cm]{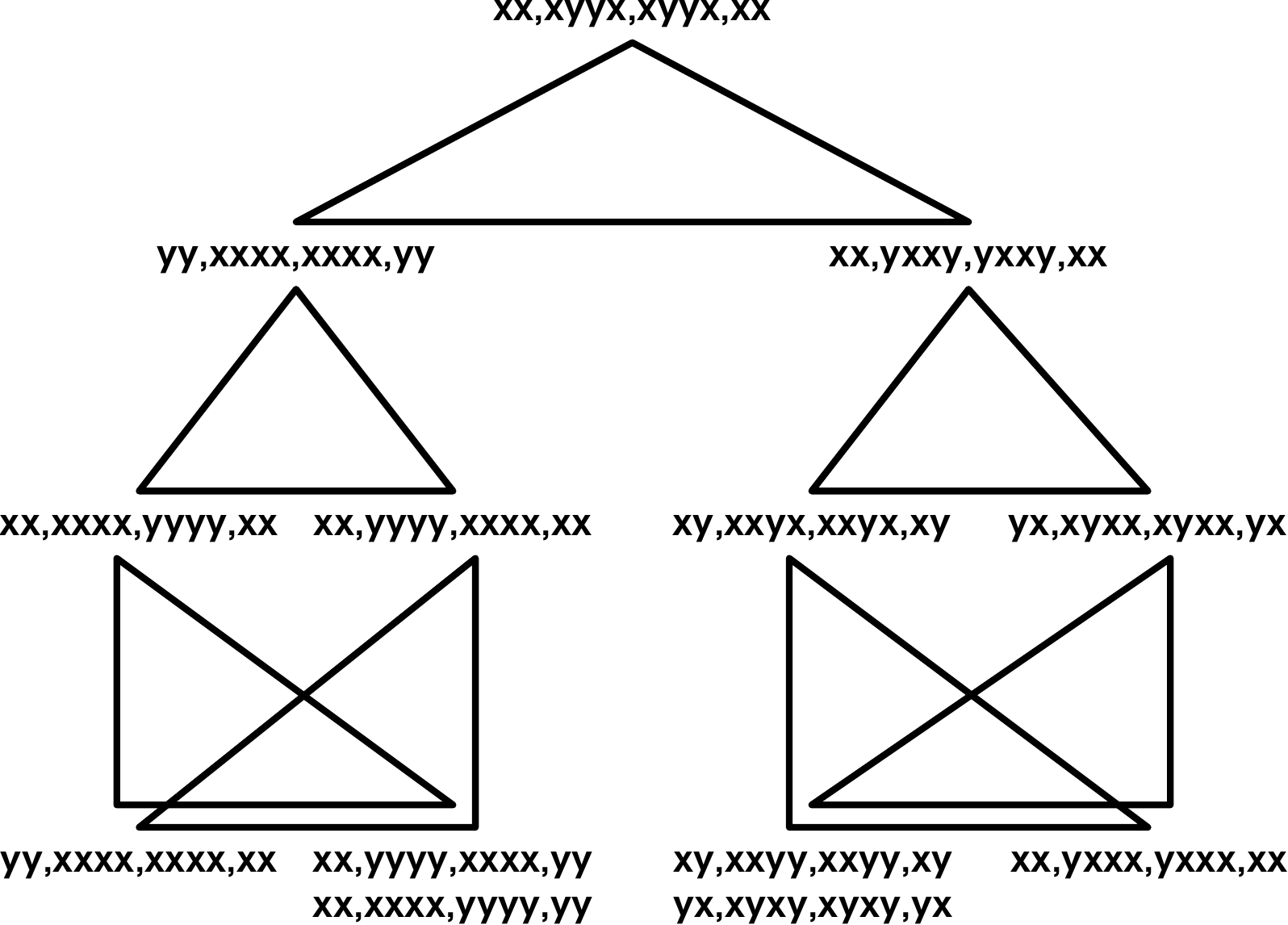}
  \caption{The derivation of terminator terms from a double loop term.}
\end{figure}

\section{Open problems} \label{sect:open}

The first area of problems is to what extend can the conditions in Theorem~\ref{thm:equiv} be further improved. P. Dapi\'c and V. Uljarevi\'c
\cite{TaylorestTerm} were able to remove 3 out of 12 columns in the double loop equations. 
The weak 3--cube term effectively removes 6 out of 12 columns, but it is not easily seen how to derive either of the two conditions from the other one. Is there a common generalization? 
A particularly interesting is the question, whether it is possible to further improve the weak 3--cube term to the so called weak 3--edge term~\cite{OptimalStrong}.
\begin{problem} Does every idempotent Taylor algebra have
  4-ary term $e$ satisfying the equations
\[
e(y,y,x,x) \approx e(y,x,y,x) \approx e(x,x,x,y)?
\]
\end{problem}
\noindent
Note that the existence of such a term follows easily from the $4$-ary Siggers term, or a 3--WNU term, or a Maltsev term.

By Theorem~\ref{thm:kazda}, single nontrivial equations do not characterize Taylor algebras. It could still be an interesting problem to compare their strength. In particular, we ask the following question.
\begin{problem}
Does every (idempotent) algebra with a $6$-ary Siggers term have a $4$-ary Siggers term?
\end{problem}

Our results hinge on the idempotency, and necessarily so by the discussion in the introduction. However, restricted classes of non-idempotent infinite algebras can posses  weakest (at least in some sense) nontrivial conditions. Of particular importance for the infinite domain CSPs is the class of closed oligomorphic algebras and its subclasses  (see eg.~\cite{BartoPinsker} for background). The following question is of interest in this context.
\begin{problem}
Let $\str A$ be a closed oligomorphic algebra that satisfies a nontrivial \emph{linear} equational condition. 
Does $\str A$ have necessarily terminator terms?
\end{problem}
\noindent
A simple example of a closed oligomorphic algebra which does not have a double loop term but has a terminator term is the algebra whose universe is a countably infinite set and the basic operations are all the injective operations. Let us also remark that the linearity assumption cannot be omitted in the problem. This will be shown in a forthcoming paper.


Our final questions are whether the NU loop lemma holds under weaker structural or algebraic assumptions, like in the finite case. An optimistic structural weakening is the following.
\begin{problem}
Let $A$ be a set,
$R\subset A^2$ a binary relation containing a finite smooth directed graph of algebraic length one (see the remarks below Theorem~\ref{thm:loop_HN} for definitions), and $f$ an NU operation on $A$ compatible with $R$.
Does $R$ necessarily contain a loop?
\end{problem}
\noindent
With a help of computer, a positive answer to this problem was verified in the case that $f$ is ternary and the finite smooth subgraph of algebraic length one has at most 4 vertices.
Also note that the assumption cannot be further weakened to ``$R$ is a smooth directed graph of algebraic length one''; a simple counterexample is the strict linear order on integers which is compatible with the median operation. 

Recall that the compatibility with an NU term cannot be weakened to the compatibility with a Taylor term, again, by Theorem~\ref{thm:kazda}. However, we do not have a counterexample to, eg., the following ``local'' version.
\begin{problem}
Let $A$ be a set,
$R\subset A^2$ a binary symmetric relation containing an odd cycle $(a_1, \dots, a_l)$, and $f$ an idempotent operation on $A$ compatible with $R$ such that, $a_1^{+R}$ absorbs $\{a_1\}$ wrt.\ $f$.
Does $R$ necessarily contain a loop?
\end{problem}
\noindent
We can prove the existence of a loop if the length $l$ of the cycle equals three.
A suitable local version of the loop lemma could help in proving a local version of the double loop lemma and this in turn may help in addressing some problems from~\cite{BartoPinsker}.

\section*{Acknowledgement}
The author gratefully acknowledges the support of the Grant Agency of the Czech Republic,
grant GA\v CR 13-01832S,
and also wishes to express his thanks to Libor Barto for his help with the article.

\bibliographystyle{plain}
\bibliography{bib-file.bib}

\begin{thebibliography}{10}

\bibitem{CSPSurvey}
Libor Barto.
\newblock The constraint satisfaction problem and universal algebra.
\newblock {\em The Bulletin of Symbolic Logic}, 21:319--337, 9 2015.

\bibitem{cyclic}
Libor Barto and Marcin Kozik.
\newblock Absorbing subalgebras, cyclic terms, and the constraint satisfaction
  problem.
\newblock {\em Log. Methods Comput. Sci.}, 8(1):1:07, 27, 2012.

\bibitem{AbsorptionSurvey}
Libor Barto and Marcin Kozik.
\newblock Absorption in universal algebra and csp.
\newblock preprint, 2016.

\bibitem{TheLoopLemma}
Libor Barto, Marcin Kozik, and Todd Niven.
\newblock The {CSP} dichotomy holds for digraphs with no sources and no sinks
  (a positive answer to a conjecture of {B}ang-{J}ensen and {H}ell).
\newblock {\em SIAM J. Comput.}, 38(5):1782--1802, 2008/09.

\bibitem{BartoPinsker}
Libor Barto and Michael Pinsker.
\newblock The algebraic dichotomy conjecture for infinite domain constraint
  satisfaction problems.
\newblock to appear in LICS'2016, 2016.

\bibitem{Bergman}
Clifford Bergman.
\newblock {\em Universal algebra}, volume 301 of {\em Pure and Applied
  Mathematics (Boca Raton)}.
\newblock CRC Press, Boca Raton, FL, 2012.
\newblock Fundamentals and selected topics.

\bibitem{BodirskySurvey}
Manuel Bodirsky.
\newblock Constraint satisfaction problems with infinite templates.
\newblock In Heribert Vollmer, editor, {\em Complexity of Constraints (a
  collection of survey articles)}, volume 5250 of {\em Lecture Notes in
  Computer Science}, pages 196--228. Springer, 2008.

\bibitem{Bodirsky-HDR-v5}
Manuel Bodirsky.
\newblock Complexity classification in infinite-domain constraint satisfaction.
\newblock M\'emoire d'habilitation \`a diriger des recherches, Universit\'{e}
  Diderot -- Paris 7. Available at arXiv:1201.0856v5, 2012.

\bibitem{BPP-projective-homomorphisms}
Manuel Bodirsky, Michael Pinsker, and Andr\'{a}s Pongr\'acz.
\newblock Projective clone homomorphisms.
\newblock Preprint arXiv:1409.4601, 2014.

\bibitem{BJK}
Andrei Bulatov, Peter Jeavons, and Andrei Krokhin.
\newblock Classifying the complexity of constraints using finite algebras.
\newblock {\em SIAM J. Comput.}, 34(3):720--742, 2005.

\bibitem{BulatovLoop}
Andrei~A. Bulatov.
\newblock {$H$}-coloring dichotomy revisited.
\newblock {\em Theoret. Comput. Sci.}, 349(1):31--39, 2005.

\bibitem{BS}
Stanley Burris and H.~P. Sankappanavar.
\newblock {\em A course in universal algebra}, volume~78 of {\em Graduate Texts
  in Mathematics}.
\newblock Springer-Verlag, New York-Berlin, 1981.

\bibitem{TaylorestTerm}
Petar Dapi\'c and Vlado Uljarevi\'c.
\newblock A note on the taylorest term of them all.
\newblock manuscript, 2016.

\bibitem{GarciaTaylor}
O.~C. Garc{\'{\i}}a and W.~Taylor.
\newblock The lattice of interpretability types of varieties.
\newblock {\em Mem. Amer. Math. Soc.}, 50(305):v+125, 1984.

\bibitem{HellNesetril}
Pavol Hell and Jaroslav Ne{\v{s}}et{\v{r}}il.
\newblock On the complexity of {$H$}-coloring.
\newblock {\em J. Combin. Theory Ser. B}, 48(1):92--110, 1990.

\bibitem{HobbyMcKenzie}
David Hobby and Ralph McKenzie.
\newblock {\em The structure of finite algebras}, volume~76 of {\em
  Contemporary Mathematics}.
\newblock American Mathematical Society, Providence, RI, 1988.

\bibitem{OptimalStrong}
Keith Kearnes, Petar Markovi{\'c}, and Ralph McKenzie.
\newblock Optimal strong {M}al'cev conditions for omitting type 1 in locally
  finite varieties.
\newblock {\em Algebra Universalis}, 72(1):91--100, 2014.

\bibitem{WNU}
Mikl{\'o}s Mar{\'o}ti and Ralph McKenzie.
\newblock Existence theorems for weakly symmetric operations.
\newblock {\em Algebra Universalis}, 59(3-4):463--489, 2008.

\bibitem{Covering}
Ralph McKenzie.
\newblock On the covering relation in the interpretability lattice of
  equational theories.
\newblock {\em algebra universalis}, 30(3):399--421, 1993.

\bibitem{Noncovering}
Ralph McKenzie and Stanislaw {\'{S}}wierczkowski.
\newblock Non-covering in the interpretability lattice of equational theories.
\newblock {\em algebra universalis}, 30(2):157--170, 1993.

\bibitem{MMT}
Ralph~N. McKenzie, George~F. McNulty, and Walter~F. Taylor.
\newblock {\em Algebras, lattices, varieties. {V}ol. {I}}.
\newblock The Wadsworth \& Brooks/Cole Mathematics Series. Wadsworth \&
  Brooks/Cole Advanced Books \& Software, Monterey, CA, 1987.

\bibitem{Siggers}
Mark~H. Siggers.
\newblock A strong {M}al'cev condition for locally finite varieties omitting
  the unary type.
\newblock {\em Algebra Universalis}, 64(1-2):15--20, 2010.

\bibitem{Taylor77}
Walter Taylor.
\newblock Varieties obeying homotopy laws.
\newblock {\em Canad. J. Math.}, 29(3):498--527, 1977.

\bibitem{Taylor88}
Walter Taylor.
\newblock Some very weak identities.
\newblock {\em Algebra Universalis}, 25(1):27--35, 1988.

\end{thebibliography}

\end{document}